\documentclass[dvi2pdf,epstopdf,12pt,a4paper,twoside]{amsart}

\usepackage{amsfonts}
\usepackage{amsmath}
\usepackage{amssymb}
\usepackage{amsthm}
\usepackage{caligr}
\usepackage{dsfont}
\usepackage{empheq}
\usepackage{enumerate}
\usepackage{enumitem}
\usepackage[multiple,symbol,stable]{footmisc}
\usepackage[hmargin={30mm,25mm},vmargin={30mm,30mm},includefoot]{geometry}
\usepackage{graphics} 
\usepackage{graphicx}
\usepackage[colorlinks]{hyperref}
\hypersetup{linkcolor=black, citecolor=black, urlcolor=black}
\usepackage{indentfirst}
\usepackage{latexsym}
\usepackage{parskip}
\usepackage{tikz}
\usetikzlibrary{positioning,fit,calc}
\usepackage[all]{xy}

\numberwithin{equation}{section}

\makeatletter
\renewcommand{\subsection}{\@startsection
{subsection}{2}{0mm}{\baselineskip}{-0.25cm}
{\normalfont\normalsize\em}}
\makeatother

\makeatletter
\def\gaps{\mathop{\operator@font Gaps}\nolimits}
\def\:={\mathrel{\mathop:}=}
\def\=:{=\mathrel{\mathop:}}
\makeatother

\def\neg1{\text{\boldmath$1$}}

\def\cV{\mathcal V}

\def\NN{\mathds{N}}
\def\NNo{\mathds{N}_0}
\def\ZZ{\mathds{Z}}

\def\ss{\caligr S}

\def\vv{\caligr V}

\newtheorem{theorem}{Theorem}[section]
\newtheorem{proposition}[theorem]{Proposition}
\newtheorem{corollary}[theorem]{Corollary}
\newtheorem{lemma}[theorem]{Lemma}

{\theoremstyle{definition}
\newtheorem{definition}[theorem]{Definition}
\newtheorem{example}[theorem]{Example}

}

\theoremstyle{remark}

\let\emptyset\varnothing

\makeatletter
\def\moverlay{\mathpalette\mov@rlay}
\def\mov@rlay#1#2{\leavevmode\vtop{%
   \baselineskip\z@skip \lineskiplimit-\maxdimen
   \ialign{\hfil$\m@th#1##$\hfil\cr#2\crcr}}}
\newcommand{\charfusion}[3][\mathord]{
    #1{\ifx#1\mathop\vphantom{#2}\fi
        \mathpalette\mov@rlay{#2\cr#3}
      }
    \ifx#1\mathop\expandafter\displaylimits\fi}
\makeatother

\newcommand{\cupdot}{\charfusion[\mathbin]{\cup}{\cdot}}
\newcommand{\bigcupdot}{\charfusion[\mathop]{\bigcup}{\cdot}}
\newcommand{\va}[1]{\left|\hspace{.02in}#1\hspace{.02in}\right|}

\makeatletter
\renewcommand*\env@matrix[1][*\c@MaxMatrixCols c]{%
  \hskip -\arraycolsep
  \let\@ifnextchar\new@ifnextchar
  \array{#1}}
\makeatother

\begin{document}

\author[G. Tizziotti]{G. Tizziotti}
\author[J. Villanueva]{J. Villanueva}





    \title[On $\kappa$-Sparse Numerical Semigroups]{On $\kappa$-Sparse Numerical Semigroups}

\address{Faculdade de Matem\'atica, Universidade Federal de Uberl\^andia, Av. Jo\~ao Naves de \'Avila 2160, 38408-100, Uberl\^andia - MG, Brasil}


\address{Instituto de Ci\^encias Exatas e da Terra, Campus Universit\'ario do Araguaia,  Universidade Federal de Mato Grosso, Rodovia MT-100, Km 3,5, 78698-000, Pontal do Araguaia - MT, Brasil}

\email{guilherme@famat.ufu.br}
\email{vz\_juan@yahoo.com.br}

    \begin{abstract}
Given a positive integer $\kappa$, we investigate the class of numerical semigroups verifying the property that every two subsequent non gaps, smaller than the conductor, are spaced by at least $\kappa$. These semigroups will be called {\em $\kappa$-sparse} and generalize the concept of sparse numerical semigroups.
    \end{abstract}

\maketitle

    \section*{Introduction}

Let $\ZZ$ be the set of integers numbers and $\NNo$ be the set of non-negative integers. A subset $H=\big\{0=n_{0}(H)<n_{1}(H)<\cdots\big\}$ of $\NNo$ is a {\em numerical semigroup} if its is closed respect to addition and its complement in $\NNo$ is finite. The numerical semigroup $\NNo$ is called {\em trivial numerical semigroup}. The cardinality $g=g(H)$ of the set $\NNo\setminus H$ is is called {\em genus} of the numerical semigroup $H$. Note that $g(H)=0$ if and only if $H=\NNo$. If $g>0$ the elements the complement $\NNo\setminus H$ are called {\em gaps} and the set of gaps will be denoted by $\gaps(H)$. The smallest integer $c=c(H)$ such that $c+h\in H$, for all $h\in \NNo$ is called the {\em conductor} of $H$. The least positive integer $n_{1}(H)\in H$ is called the {\em multiplicity} of $H$. As $\NNo\setminus H$ is finite, the set $\ZZ\setminus H$ has a maximum, which is called {\em Frobenius number} and will be denoted by $\ell_{g}(H)$. A property known of this number is that $\ell_{g}(H)\leq 2g-1$, see \cite{oliveira}. In particular, $H=\NNo$ if and only if $-1$ is the Frobenius number of $H$. At consequence of this fact, from now on we use the notation $\ell_{0}=\ell_{0}(H)\:=-1$, for all numerical semigroup $H$. When $g>0,$ we denote $\gaps(H)=\big\{1=\ell_{1}(H)<\cdots<\ell_{g}(H)\big\}$. So, $c=\ell_{g}(H)+1$ and is clear that $c=n_{c-g}(H)$. For simplicity of notation we shall write $\ell_{i}$ for $\ell_{i}(H)$ and $n_{k}$ for $n_{k}(H)$, for all integers $i,k$ such that $0\leq i\leq g$ and $k\geq0$, when there is no danger of confusion.

Let $H=\{0=n_{0}<n_{1}<\cdots\}$ be a numerical semigroup of genus $g>0$ with $\gaps(H)=\{\ell_{1}<\cdots<\ell_{g}\}$. $H$ is said to be {\em sparse numerical semigroup} if
   \begin{equation}\label{eq1.1}
\ell_{i}-\ell_{i-1}\leq 2,
   \end{equation}
for all integer $i$ such that $1\leq i\leq g$, or equivalently $n_{i}-n_{i-1}\geq 2$, for all integer $i$ such that $1\leq i\leq c-g$. The concept of a sparse numerical semigroup was introduced in \cite{MTV} and \cite{JV}. For convenience, we consider the numerical semigroup $\NNo$ as sparse.
For each non-negative integer $g$, let
    $$
\NN_{g}\:=\{0\}\cup\{n\in\NN:\,n\geq g+1\}
    $$
(in the case $g=0$, it notation is itself the $\NNo$). It is clear that $\NN_{g}$ is a numerical semigroup of genus $g$. The semigroup $\NN_{g}$ is called {\em ordinary numerical semigroup} and is a canonical examples of sparse numerical semigroups. In this work, we introduce and investigate the class of numerical semigroups verifying the property that every two subsequent non gaps, smaller than the conductor, are spaced by at least $\kappa$, where $\kappa$ is a positive integer. These semigroups will be called {\em $\kappa$-sparse} and generalize the concept of sparse numerical semigroups.

This paper is organized as follows. The Arf numerical semigroups was introduced in \cite{arf} and since then several equivalent properties for these numerical semigroups have appeared. In the Section 1, we present one of these properties, given in \cite{cfm}, and prove that it is equivalent to a characterization given in \cite{barucci}. The Section 2, contains the concept of leaps on a numerical semigroup which serves as the basis for introducing the concept of $\kappa$-sparse numerical semigroup that will be presented in the Section 3. Finally, in the Section 4, we will present results on the structures of $\kappa$-sparse numerical semigroups.

    \section{Arf Numerical Semigroups}\label{s1}

    \begin{definition} \label{def2.1}
A semigroup $H=\{0=n_{0}<n_{1}<\cdots\}$ is called {\em Arf numerical semigroup} if
   \begin{equation}\label{eq2.1}
n_{i}+n_{j}-n_{k}\in H,
   \end{equation}
for all integers $i,j,k$ such that $0\leq k\leq j\leq i$.

That is equivalent to $n_{i}+n_{j}-n_{k}\in H$, for all integers $i,j,k$ such that $0\leq k\leq j\leq i\leq c-g$.
\end{definition}
Examples of Arf numerical semigroups are ordinary numerical semigroups. Relevant examples of sparse numerical semigroups are Arf numerical semigroups.

For more details about Arf semigroups see e.g. \cite{barucci} and \cite{rosales}. The Arf property (\ref{eq2.1}) is equivalent to: $2n_{i}-n_{j}\in H$, for all integers $i,j$ such that $0\leq j\leq i$ (cf. \cite{cfm}). In \cite{barucci}, there is fifteen alternative characterizations of Arf numerical semigroups which are distinct of the given in the Definition \ref{def2.1}. The goal of this section is to prove that the property of Arf given in Equation \ref{eq2.1} is equivalent to the characterization given in the item (xv)\footnote{Theorem I.3.4 (xv) : $H(n_{k})$ is stable ideal$,$ for all integer $k$ such that $1\leq k\leq c-g.$} of \cite[Theorem I.3.4]{barucci}. Before describe of this proof we recall some basics concepts, more details can be found in \cite{barucci}. For two sets $E$ and $F$ contained in $\ZZ$ and $z\in\ZZ$, let
    $$
E+F\:=\{e+f:\,e\in E,f\in F\}\quad\text{and}\quad E+z\:=\{e+z:\,e\in E\}.
    $$
Let $H$ be a semigroup numerical. A {\em relative ideal} of $H$ is a nonempty subset $E$ of $\ZZ$ such that $E+H\subset E$ and $E+h\subset H$, for some $h\in H$. A relative ideal of $H$ which is contained in $H$ is simply called an {\em ideal} of $H$. A {\em proper ideal} is an ideal distinct from $H$, that is, an ideal that does not contain zero. By a {\em principal ideal} of $H$, we mean relative ideal $E$ such that $E=H+z$, for some $z\in\ZZ$. In this case, $E$ is called an {\em relative ideal principal of $H$ generated by $z$}.

If $E$ and $F$ are relative ideal of $H$ then $(E-F)\:=\{z\in\ZZ:\,F+z\subset E\}$ is also a relative ideal of $H$. For every relative ideal $E$ of $H$, $(E-E)$ is a numerical semigroup.

A ideal $E$ of $H$ is called {\em stable} if $E$ is a principal ideal of $(E-E)$.

Let $H=\{0=n_{0}<n_{1}<\cdots\}$ be a numerical semigroup. For each $i\in\NNo$, let
    $$
H(n_{i})=\{h\in H:\,h\geq n_{i}\}.
    $$
Is clear that $H(n_{i})$ is an ideal of $H$, for all $i\in\NNo$. If $i\neq0$ then $H(n_{i})$ is an proper ideal of $H$.

    \begin{theorem}\label{thm2.2}
Let $H=\{0=n_{0}<n_{1}<\cdots\}$ be a numerical semigroup of genus $g$ and conductor $c.$ The following statements are equivalent$:$
    \begin{enumerate}[label=$(\arabic*)$,leftmargin=*,align=left,start=1,itemsep=0pt]
\item $H$ is an Arf numerical semigroup$;$
\item $H(n_{k})$ is stable ideal$,$ for all $k\in\NN;$
\item $H(n_{k})$ is stable ideal$,$ for all integer $k$ such that $1\leq k\leq c-g.$
    \end{enumerate}
    \end{theorem}

    \begin{proof}
We prove that $(1)\Rightarrow(2)\Rightarrow(3)\Rightarrow(1)$. Initially suppose that $H$ be a Arf numerical semigroup and let $k$ be a positive integer.  We must show that $H(n_{k})=\big(H(n_{k})-H(n_{k})\big)+n_{k}$. Since $n_{k}\in H(n_{k})$, it is clear that $\big(H(n_{k})-H(n_{k})\big)+n_{k} \subset H(n_{k})$. Now, let $h \in H(n_{k})$. So, $h=n_{j}$ for some $j\geq k$. Since $n_{j}=(n_{j}-n_{k})+n_{k}$, in order to show $H(n_{k})\subset\big(H(n_{k})-H(n_{k})\big)+n_{k}$ is enough prove that  $n_{j}-n_{k}\in\big(H(n_{k})-H(n_{k})\big)$. Since $H$ is Arf, we have that $n_{i}+n_{j}-n_{k}\in H$, for all $i,j \geq k$. Thus, $n_{i}+n_{j}-n_{k}\in H(n_{k})$, for all integer $i$ such that $i\geq k$. Therefore, $H(n_{k})+(n_{j}-n_{k})\subset H(n_{k})$ and follows that $n_{j}-n_{k}\in\big(H(n_{k})-H(n_{k})\big)$. Thus, we have $(1)\Rightarrow(2)$. The implication $(2)\Rightarrow(3)$ is trivial.

Now, suppose $(3)$. Let $i,j,k$ integers such that $0\leq k\leq j\leq i\leq c-g$. If $k=0$ is clear that $n_{i}+n_{j}-n_{k}\in H$. So, let $k\geq1$ and suppose that $H(n_{k})$ is stable. Then,  $H(n_{k})=\big(H(n_{k})-H(n_{k})\big)+z$, for some $z\in\ZZ$. Note that $z\in H(n_{k})$, since $0\in\big(H(n_{k})-H(n_{k})\big)$. Also, $n_{k}=x+z$, for some $x\in\big(H(n_{k})-H(n_{k})\big)$. Then, $n_{k}=x+z\geq x+n_{k}\geq n_{k}$. Therefore, $x=0$ and we have that $H(n_{k})=\big(H(n_{k})-H(n_{k})\big)+n_{k}$. Let $i,j,k$ be integers such that $1 \leq k \leq j \leq i$. Since $n_{j}\in H(n_{k})$, we have that $n_{j}=y+n_{k}$, for some $y\in\big(H(n_{k})-H(n_{k})\big)$. Consequently, $n_{i}+n_{j}-n_{k}=n_{i}+y\in H(n_{i})+y\subset H(n_{k})+y\subset H(n_{k})\subset H$ and thus $(3)\Rightarrow(1)$.
   \end{proof}

    \section{The sets of leaps}\label{s2}

Let $H$ be a numerical semigroup of genus $g>0$ and $\gaps(H)=\{1=\ell_{1}<\cdots<\ell_{g}\}$. Remembering that $\ell_{0}=-1$, for each $1\leq i\leq g$, the ordered pair $(\ell_{i-1},\ell_{i})$ will be called {\em leap on $H$} (or simply {\em leap}). The set of leaps on $H$ will be denoted by
    $$
\cV=\cV(H)\:=\big\{(\ell_{i-1},\ell_{i}):\,1\leq i\leq g\big\}.
    $$
Note that $\va{\cV}=g$.

In \cite{contiero-gugu-paula}, the authors work with the ordered pairs $(\ell_{i-1},\ell_{i})$ such that $\ell_{i}-\ell_{i-1}=1$ and $\ell_{i}-\ell_{i-1}=2$, for sparse numerical semigroups of genus $g \geq 2$ and $2 \leq i \leq g$, not considering the element $\ell_{0}=-1$. These pairs were called {\em single leap} and {\em double leap}, respectively.

Based on this idea, for a positive integer $m$, if $\ell_{i}-\ell_{i-1}=m$ the ordered pair $(\ell_{i-1},\ell_{i})$ will be called {\em $m$-leap on $H$} (or simply {\em $m$-leap}), for any numerical semigroups with genus $g>0$ and $i=1,\dots,g$. If $m=1$, the $1$-leap on $H$ is called a {\em single leap on $H$} (or simply {\em single leap}). If $m=2$, the $2$-leap on $H$ is called a {\em double leap on $H$} (or simply {\em double leap}). For a positive integer $m$, the set of $m$-leaps will be denoted by
    $$
\cV_{m}=\cV_{m}(H)\:=\big\{(\ell_{i-1},\ell_{i}):\,\ell_{i}-\ell_{i-1}=m,\;1\leq i\leq g\big\}.
    $$
For convenience, we define $\cV_{m}(\NNo)\:=\emptyset$, for all positive integer $m$. To simplify the notation we will denote the cardinality of the set $\cV_{m} (H)$ by $v_{m}(H)$, that is, $v_{m}=v_{m}(H)\:=\va{\cV_{m}(H)}$.

From definition the sets $\cV$ and $\cV_{m}$'s, for all numerical semigroup $H$, we get the disjoint union
    $$
\cV(H)=\bigcupdot_{m\in\NN}\cV_{m}(H).
    $$

Observe that, if $H$ is a numerical semigroup of genus $g>0, then$ $\cV(H)$ is a finite union of non-empty disjoint sets$.$

Before of the next result, let us remember that a numerical semigroup $H$ is called {\em hyperelliptic} if $2 \in H$.

    \begin{theorem}\label{thm3.1}
Let $H$ be a numerical semigroup of genus $g>0.$ Then$,$ $\cV$ is a finite
union of non-empty sets mutually disjoint$.$ In particular$,$ $\sum\limits_{m\in\NN}v_{m}=g.$ More precisely$,$
    \begin{enumerate}[label=$(\arabic*)$,leftmargin=*,align=left,start=1,itemsep=0pt]
\item\label{k1} $H$ is hyperelliptic if and only if $v_{1}=0.$ In this case$,$ $v_{2}=g$ and $v_{m}=0,$ for all integer $m$ such that $m\geq3;$
\item\label{k2} $H$ is non-hyperelliptic if and only if $v_{1}\neq0.$ In this case$,$ $v_{g+m}=0,$ for all positive integer $m.$
    \end{enumerate}
    \end{theorem}

    \begin{proof}
Let $H$ be a numerical semigroup of genus $g>0$ with $\gaps(H)=\{\ell_{1}<\dots<\ell_{g}\}$.

Let us prove  the item \ref{k1}. By \cite[Theorem 1.1 (i)]{oliveira}, $H$ is hyperelliptic if and only if $\ell_{i}=2i-1$, for all integer $i$ such that $1\leq i\leq g$. So, if $H$ is hyperelliptic then $v_{1}=0$. Reciprocally, if $v_{1}=0$ then $2\in H$, otherwise we will have $\ell_{1}=1$ and $\ell_{2}=2$, and so $(\ell_{1},\ell_{2})$ is a single leap, contradiction. Moreover, since $\ell_{i}=2i-1$, for all integer $i$ such that $1\leq i\leq g$, follows that $\cV_{2}=\cV$ and $\cV_{m}=\emptyset$, for all integer $m$ such that $m\geq3$ and the assertion \ref{k1} follows.

The first assertion in the item \ref{k2}, follows directly from the item \ref{k1}, that is, $H$ is non-hyperelliptic if and only if $v_{1}\neq0$. Now, we prove the last assertion. By \cite[Theorem 1.1 (ii)]{oliveira}, $H$ is non-hyperelliptic if and only if $H=\NN_{2}$ or $g\geq3$ and $\ell_{i}\leq 2i-2$, for all integer $i$ such that $2\leq i\leq g-1$, and $\ell_{g}\leq 2g-1$. Thus, if $H=\NN_{2}$ then $g=2$ and we have $\cV=\big\{(\ell_{1},\ell_{2})\big\}=\cV_{1}$. So, $\cV_{2+m}=\emptyset$, for all positive integer $m$. On the other hand, if $g\geq 3$ then $\ell_{i}-\ell_{i-1}\leq(2i-1)-(i-1)\leq g-2$, for all integer $i$ such that $2\leq i\leq g-1$, and $\ell_{g}-\ell_{g-1}\leq(2g-1)-(g-1)=g$. This implies that $\cV_{g+m}=\emptyset$, for all positive integer $m$, and the proof is complete.
    \end{proof}

    \begin{proposition}\label{prop3.2}
Let $H$ be a numerical semigroup of genus $g.$ Then$,$
    \begin{enumerate}[label=$(\arabic*)$,ref=$(\arabic*)$,leftmargin=*,align=left,start=1,itemsep=0pt]
\item\label{l1} $H=\NNo$ if and only if $v_{2}(H)=0;$
\item\label{l2} $v_{2}(H)\neq 0$ if and only if $1\not\in H;$
\item\label{l3} for $g>0,$ $H=\NN_{g}$ if and only if $v_{1}(H)=g-1$ and $v_{2}(H)=1.$ In this case$,$ $v_{m}=0,$ for all positive integer $m\geq3;$
\item\label{l4} for all positive integer $\kappa,$ $\sum\limits_{m=1}^{\kappa}v_{m}(H)\leq g.$
    \end{enumerate}
    \end{proposition}

    \begin{proof}
To prove \ref{l1}, by definition, we have that $\cV_{m} (\NNo)=\emptyset$, for all positive integer $m$. So, $H=\NNo$ implies that $v_{2}(H)=0$. Now, suppose that $v_{2}(H)=0$. Then, $1\in H$ and follows that $H=\NNo$. The item \ref{l2}, follows directly from the previous item.

For the item \ref{l3}, it is clear that if $H=\NN_{g}$, then $\ell_{i}=i$, for all integer $i$ such that $1\leq i\leq g$. So, $v_{1}(H)=g-1$ and $v_{2}(H)=1$. Reciprocally, suppose $v_{1}(H)=g-1$ and $v_{2}(H)=1$. If $g=1$, we have that $v_{1}(H)=0$ and, by item (1) from the Theorem \ref{thm3.1}, follows that $H=\NN_{1}$. Then, suppose $g\geq 2$. Since $|\cV(H)|=g\geq2$ and $v_1(H)=g-1$, $(\ell_{0},\ell_{1})$ is the unique double leap. So, $\cV_{1}(H)=\cV (H)\setminus\big\{(\ell_{0},\ell_{1})\big\}$ and $\ell_{i}-\ell_{i-1}=1$, for all integer $i$ such that $2\leq i\leq g$. Therefore, we conclude that $\ell_{i}=i$, for all integer $i$ such that $1\leq i\leq g$, and we have $H=\NN_{g}$.

The item \ref{l4} is trivial, since $\bigcupdot\limits_{m=1}^{\kappa}\cV_{m}(H)\subset\cV(H)$, for all positive integer $\kappa$.
\end{proof}

    \begin{theorem}\label{thm3.3}
Let $H$ be a numerical semigroup of genus $g.$
    \begin{enumerate}[label=$(\arabic*)$,ref=$(\arabic*)$,leftmargin=*,align=left,start=1,itemsep=0pt]
\item\label{sparse1} $H$ is a sparse numerical semigroup if and only if $v_{1}(H)+v_{2}(H)=g.$ In this case$,$ $v_{m}(H)=0,$ for all positive integer $m\geq3.$
\item\label{sparse2} If $H$ is a sparse numerical semigroup$,$ then the Frobenius number of $H$ is $\ell_{g}=v_{1}(H)+2v_{2}(H)-1.$
\item\label{sparse3} If $g>0$ and $\ell_{g}(H)=2g-K,$ for some positive integer $K,$ then $H$ is a sparse numerical semigroup if and only if $v_{1}(H)=K-1$ and $v_{2}(H)=g-K+1.$
    \end{enumerate}
    \end{theorem}

    \begin{proof}
To prove \ref{sparse1} first, suppose that $H$ is a sparse numerical semigroup. Then, $\ell_{i}-\ell_{i-1}\leq 2$, for all $1\leq i\leq g$. Since $\va{\cV(H)}=g$, we have that $v_{1}(H)+v_{2}(H)=g$. The converse implication follows directly from $\cV_{1}(H)\cupdot\cV_{2}(H)\subset\cV(H)$ and $\va{\cV(H)}=g$.

Now, we prove \ref{sparse2}. If $H=\NNo$, then $g=0$ and $v_{1}(H)=v_{2}(H)=0$, and the equality is true since the Frobenius number of $H=\NNo$ is $\ell_{0}=-1$. Suppose that $g>0$. If $H$ is hyperelliptic, by Theorem \ref{thm3.1} item (1), we have that $v_{1}(H)=0$ and $v_{2}(H)=g$. So, the equality follows, since $\ell_{g}=2g-1$. If $H$ is non-hyperelliptic, by Theorem \ref{thm3.1} item (2), $\cV_{1}(H)\neq\emptyset$. If $\gaps(H)=\{\ell_{1}<\cdots<\ell_{g}\}$, since $H$ is sparse, by Theorem \ref{thm3.3} item (1), we have that
    $$
\sum_{i=1}^{g}(\ell_{i}-\ell_{i-1})=\sum_{(\ell_{i-1},\ell_{i})\in\cV_{1}(H)}(\ell_{i}-\ell_{i-1})+\sum_{(\ell_{i-1},\ell_{i})\in\cV_{2}(H)}(\ell_{i}-\ell_{i-1}).
    $$
Note that $\ell_{g}-\ell_{0}=\sum\limits_{i=1}^{g}(\ell_{i}-\ell_{i-1})$. By the definitions of $\cV_1(H)$ and $\cV_2(H)$, follows that $\sum\limits_{(\ell_{i-1},\ell_{i})\in\cV_1}(\ell_{i}-\ell_{i-1})=v_1(H)$ and $\sum\limits_{(\ell_{i-1},\ell_{i})\in\cV_2}(\ell_{i}-\ell_{i-1})=2v_2(H)$. Therefore, $\ell_{g}-\ell_{0}=v_1(H)+2v_2(H)$ by proving the result.

Finally, we prove the item \ref{sparse3}. Suppose that $g>0$ and $\ell_{g}(H)=2g-K$. If $H$ is a sparse numerical semigroup, by the item \ref{sparse1}, we have that $v_{1}(H)+v_{2}(H)=g$ and, by the item \ref{sparse2}, $\ell_{g}=v_{1}(H)+2v_{2}(H)-1$. So, $2g-K=g+v_{2}(H)-1$ implies $v_{2}(H)=g-K+1$ and, using the item \ref{sparse1} again, $v_{1}(H)=g-1$. The converse implication follows directly by the item \ref{sparse1} above.
    \end{proof}

    \section{$\kappa$-Sparse Numerical Semigroup}\label{s3}

Let $H$ be a numerical semigroup of genus $g$. The Theorem \ref{thm3.1} shows that $\cV_{g+m}=\emptyset$, for all integer $m$ such that $m\geq2$, and
    $$
\sum_{m=1}^{g+1}v_{m}(H)=g.
    $$
The item (1) in Theorem \ref{thm3.3} motivates us to define a new class of numerical semigroups, which generalize the sparse numerical semigroups.

    \begin{definition}\label{def4.1}
Let $\kappa$ be an integer positive. A numerical semigroup $H$ of genus $g$ is called {\em $\kappa$-sparse numerical semigroup} if
    $$
\sum_{m=1}^{\kappa}v_{m}(H)=g.
    $$
    \end{definition}

Note that if $H$ is a $\kappa$-sparse numerical semigroup, then $v_{k+j}(H)=0$, for all positive integer $j$. In particular, $\cV(H)=\bigcupdot\limits_{m=1}^{\kappa}\cV_{m}(H)$.

For $\kappa \geq 2$, a $\kappa$-sparse numerical semigroup $H$ with $v_{\kappa}(H) \neq 0$ will be called \emph{pure $\kappa$-sparse numerical semigroup}.

By the definition above, $\NNo$ is a $\kappa$-sparse numerical semigroup, for all positive integer $\kappa$, and the sparse numerical semigroups are the $2$-sparse numerical semigroups. Also, the pure $2$-sparse numerical semigroups are the sparse numerical semigroups with genus positive.

    \begin{example}\label{ex4.2}
For each non-negative integer $g$, the ordinary numerical semigroup $\NN_{g}$ is a $\kappa$-sparse numerical semigroup, for all integer $\kappa$ such that $\kappa\geq2$.
    \end{example}

The definition of $\kappa$-sparse condition is equivalent a another properties as follows, which is our main result.

    \begin{theorem}\label{thm4.3}
Let $\kappa$ be an integer such that $\kappa\geq2.$ Let $H=\{0=n_{0}<n_{1}<\cdots\}$ be a numerical semigroup of genus $g>0$ and $\gaps(H)=\{1=\ell_{1}<\cdots<\ell_{g}\},$ with conductor $c.$ The following statements are equivalent$:$
    \begin{enumerate}[label=$(\arabic*)$,leftmargin=*,align=left,start=1,itemsep=0pt]
\item\label{s1} $H$ is be a $\kappa$-sparse numerical semigroups$;$
\item\label{s2} $\ell_{i}-\ell_{i-1}\leq\kappa,$ for all integer $i$ such that $1\leq i\leq g;$
\item\label{s3} $n_{i+\kappa-2}-n_{i-1}\geq\kappa,$ for all integer $i$ such that $1\leq i\leq c-g-\kappa+2;$
\item\label{s4} if $i$ is a positive integer such that $n_{i},n_{i}+1,\dots,n_{i}+\kappa-1\in H,$ then $n_{i}\geq c$ $($that is$,$ $i\geq c-g).$
    \end{enumerate}
    \end{theorem}

    \begin{proof}
We prove that $(1)\Rightarrow(2)\Rightarrow(3)\Rightarrow(4)\Rightarrow(1)$. Initially suppose that $H$ be a $\kappa$-sparse numerical semigroup. Then, $\cV(H)=\bigcupdot\limits_{m=1}^{\kappa}\cV_{m}(H)$. So, given $i\in\{1,\dots,g\}$, there exist $m\in\{1,\dots,\kappa\}$ such that $(\ell_{i-1},\ell_{i})\in\cV_{m}(H)$. Thus, $\ell_{i}-\ell_{i-1}=m\leq\kappa$ and we have $(1)\Rightarrow(2)$.

Now, suppose $(2)$. First, if $c=g+\kappa-1$, then $c-g-\kappa+2=1$ and we have that $n_{\kappa-1}-n_{0}=n_{\kappa-1}\geq\kappa$. Now, let $c\geq g+\kappa$ and suppose that exist $i\in\{1,2,\dots,c-g-\kappa+2\}$ such that $n_{i+\kappa-2}-n_{i-1}<\kappa$. Since $n_{i+\kappa-2}-n_{i-1}\geq(i+\kappa-2)-(i-1)=\kappa-1$, follows that $n_{i+\kappa-2}-n_{i-1}=\kappa-1$. Note that on the interval $[n_{i-1},n_{i+\kappa-2}]$ exists exactly $\kappa$ consecutive numbers and thus $n_{i+s-1}=n_{i-1}+s$, for all integer $s$ such that $1\leq s\leq\kappa-1$. So, $i<c-g-\kappa+2$, since $n_{c-g}=c$. Choice the largest index $i$ with that property. Then, $n_{i+\kappa-2}+1$ is a gap of $H$. That is, there exist $j\in\{1,\dots,g\}$ such that $\ell_{j}=n_{i+\kappa-2}+1$. Therefore, $\ell_{j-1}<n_{i-1}$ and so $\ell_{j}-\ell_{j-1}>(n_{i+\kappa-2}+1)-n_{i-1}=\kappa$, a contradiction. Thus, we have that $(2)\Rightarrow(3)$.

To prove the implication $(3)\Rightarrow(4)$, suppose there exist a positive integer $i < c-g$ such that $n_{i},n_{i}+1,\dots,n_{i}+\kappa-1\in H$. Note that $n_{i+\kappa-1}=n_{i}+\kappa-1$ and $i \leq c-g +\kappa +1$. Then, we have $1\leq i \leq c-g-\kappa+1$ and $n_{i+\kappa-1}-n_{i}=(n_{i}+\kappa-1)-n_{i}=\kappa-1$, a contradiction with the item $(3)$.

Finally, we proof the implication $(4)\Rightarrow(1)$. Suppose that there is a number $i\in\{1,\dots,g\}$ such that $(\ell_{i-1},\ell_{i})\in \cV(H)\setminus\bigcupdot\limits_{m=1}^{\kappa}\cV_{m}(H)$. So, $\ell_{i}-\ell_{i-1}>\kappa$. Note that $i \geq 2$, since $\kappa \geq 2$ and $g>0$. Then, $\ell_{i}-\kappa,\ell_{i}-\kappa+1,\dots,\ell_{i}-1$ are $\kappa$ consecutive non-gaps with $\ell_{i}-\kappa<\ell_{i}<c$, a contradiction with the item (4).  Therefore, $\cV(H)=\bigcupdot\limits_{m=1}^{\kappa}\cV_{m}(H)$ and we have that $H$ is a $\kappa$-sparse numerical semigroup.
    \end{proof}

    \begin{corollary}\label{crl4.4}
Let $\kappa$ be an integer such that $\kappa\geq3.$ Let $H$ be a numerical semigroup with conductor $c.$ Then$,$ $H$ is a pure $\kappa$-sparse numerical semigroup if and only if $H$ is a $\kappa$-sparse numerical semigroup and there is a positive integer $i$ such that $n_{i},n_{i}+1,\dots,n_{i}+\kappa-2\in H,$ with $n_{i}<c.$
    \end{corollary}

    \begin{proof}
Initially suppose that $H$ is a pure $\kappa$-sparse numerical semigroup. Thus, $H$ is $\kappa$-sparse numerical semigroup and there is an integer $j\in \{1,\ldots,g\}$ such that $\ell_{j}-\ell_{j-1}=\kappa$. So, there is $i$ such that $n_{i},n_{i}+1,\dots,n_{i}+\kappa-2\in[\ell_{i-1},\ell_{i}] \cap H$ and, since $\ell_{j}<c$, we have that $n_{i}<c$.

Reciprocally, suppose that $H$ is a $\kappa$-sparse and there is a positive integer $i$ such that $n_{i},n_{i}+1,\dots,n_{i}+\kappa-2\in H$, with $n_{i}< c$. By Theorem \ref{thm4.3} item (iv), follows that $n_{i}-1,n_{i}+\kappa-1\notin H$, that is, $(n_{i}-1,n_{i}+\kappa-1)\in\cV_{\kappa}(H)$ and follows that $H$ is a pure $\kappa$-sparse numerical semigroup.
    \end{proof}


    \begin{example}\label{ex4.5}
Let $\kappa$ be an integer such that $\kappa\geq3$. For an integer $a$ such that $a\geq\kappa$, considerer the numerical semigroup $H=\{0=n_{0}<n_{1}<\cdots\}$ of genus $g=2a-\kappa,$ where $n_{1}=a$ and $n_{\kappa}=2a$. Then, $H$ is be a pure $\kappa$-sparse numerical semigroup if and only if $H=\{a,a+1,\dots,a+\kappa-2\}\cup\NN_{2a-1}$.

In fact, suppose that $H$ is be a pure $\kappa$-sparse numerical semigroup. Since on the interval $[n_{1},n_{\kappa}]=[a,2a]$ has $a+1-\kappa$ gaps and $g=2a-\kappa$, follows that $2a+j\in H$, for all $j\in\NNo$. Thus, the conductor $c$ of $H$  is such that  $c\leq2a$. So, by Corollary \ref{crl4.4}, follows that $n_{i+1}=n_{1}+i$, for all integer $i$ such that $1\leq i\leq\kappa-2$. Then, $H=\{a,a+1,\dots,a+\kappa-2\}\cup\NN_{2a-1}$.

Now, suppose that $H=\{a,a+1,\dots,a+\kappa-2\}\cup\NN_{2a-1}$. Thus, if $\gaps(H)=\{\ell_{1}<\cdots<\ell_{g}\}$, then $\ell_{i}=i$ and $\ell_{j}=j+\kappa-1$, for all integers $i,j$ such that $1\leq i\leq a-1$ and $a\leq j\leq g=2a-\kappa$. So, $\ell_{i}-\ell_{i-1}=1$, for all integer $i\in\{2,\dots,g\}\setminus \{a\}$, and $\ell_{a}-\ell_{a-1}=\kappa$. Therefore, $H$ is be a pure $\kappa$-sparse numerical semigroup.
    \end{example}

    \section{Structure of $\kappa$-Sparse Numerical Semigroups}\label{s4}

For each positive integer $\kappa,$ let $\ss_{\kappa}$ be the collection of $\kappa$-sparse numerical semigroups and let $\ss_{\kappa}^{\;\ast}$ be the collection of pure $\kappa$-sparse numerical semigroups, where, for $\kappa=1$, we will denote $\ss_{1}^{\;\ast}=\ss_{1}$.

    \begin{lemma}\label{lem5.1}
\hfill
    \begin{enumerate}[label=$(\arabic*)$,leftmargin=*,align=left,start=1,itemsep=0pt]
\item $\ss_{\kappa}=\{\NNo\}$ if and only if $\kappa=1;$
\item $\ss_{\kappa}=\{\text{$H:$\,$H$ is be a sparse numerical semigroup}\}$ if and only if $\kappa=2.$
    \end{enumerate}
    \end{lemma}

    \begin{proof}
We prove (1). First, suppose that $\ss_{\kappa}=\{\NNo\}$. If $\kappa>1$ there is $H\in\ss_{\kappa}$ such that $g(H)>0$, a contradiction. Reciprocally, suppose that $\kappa=1$ and let $H\in\ss_{1}$. Then, $v_{m}(H)=0$, for all integer $m$ such that $m\geq2$. In particular, $v_{2}(H)=0$. Therefore, by Proposition \ref{prop3.2} item (1), $H=\NNo$ and so $\ss_{1}=\{\NNo\}$.

Now, we prove (2). Suppose that $\ss_{\kappa}$ be the collection of sparse numerical semigroups. From item (1), follows that $\kappa\geq2$. By Theorem \ref{thm3.3} item \ref{sparse1}, $v_{1}(H)+v_{2}(H)=g$, for all $H\in\ss_{\kappa}$. Then, $\kappa\leq2$ and we have $\kappa=2$. Reciprocally, suppose that $\kappa=2$ and let $H\in\ss_{2}$. Then, $v_{1}(H)+v_{2}(H)=g$. So, by Theorem \ref{thm3.3} item \ref{sparse1}, $H$ is be a sparse numerical semigroup and the proof is complete.
    \end{proof}

The following result is a generalization of Theorem \ref{thm3.3}.

    \begin{theorem}\label{thm5.2}
Let $\kappa$ be an integer such that $\kappa\geq2.$ Let $H$ be a numerical semigroup of genus $g.$ Then$:$
    \begin{enumerate}[label=$(\arabic*)$,ref=$(\arabic*)$,leftmargin=*,align=left,start=1,itemsep=0pt]
\item\label{ksparse1} If $H$ is be a $\kappa$-sparse numerical semigroup$,$ then the Frobenius number of $H$ is
    $$
\ell_{g}=\sum_{m=1}^{\kappa}mv_{m}(H)-1;
    $$
\item\label{ksparse2} Let $g>0$ and $K\:=2g-\ell_{g}.$ Then$,$ $H$ is be a $\kappa$-sparse numerical semigroup if and only if
    $$
\sum_{m=1}^{\kappa}mv_{m}(H)=2g-K+1\quad\text{and}\quad\sum_{m=2}^{\kappa}(m-1)v_{m}(H)=g-K+1.
    $$
    \end{enumerate}
    \end{theorem}

    \begin{proof}
We prove \ref{ksparse1}. If $\kappa=1$, by Lemma \ref{lem5.1} item (1), $H=\NNo$. Then, $g=0$ and $v_{1}(H)=0$. So, the equality is true, since the Frobenius number of $H=\NNo$ is $\ell_{0}=-1$. Suppose that $\kappa>1$ and thus $g>0$. If $H$ is hyperelliptic, by Theorem \ref{thm3.1} item (1), we have that $v_{2}(H)=g$ and $v_{m}(H)=0$, for all positive integer $m$ such that $m\neq2$. So, the equality follows, since $\ell_{g}=2g-1$ and
    $$
2g-1=2g+\sum_{\substack{{m=1}\\{m\neq 2}}}^{\kappa}mv_{m}(H)-1.
    $$
Now, if $H$ is non-hyperelliptic, by Theorem \ref{thm3.1} item (2), $\cV_{1}(H)\neq\emptyset$. Let $\gaps(H)=\{\ell_{1}<\cdots<\ell_{g}\}$. By the $\kappa$-sparse property, we have that
    $$
\sum_{i=1}^{g}(\ell_{i}-\ell_{i-1})=\sum_{\substack{{m=1}\\{\cV_{m}(H)\neq\emptyset}}}^{\kappa}\sum_{(\ell_{i-1},\ell_{i})\in\cV_{m}(H)}(\ell_{i}-\ell_{i-1}).
    $$
Since $\ell_{g}-\ell_{0}=\sum\limits_{i=1}^{g}(\ell_{i}-\ell_{i-1})$ and $\sum\limits_{(\ell_{i-1},\ell_{i})\in\cV_{m}(H)}(\ell_{i}-\ell_{i-1})=mv_{m}(H)$ (by the definition of $\cV_{m}(H)$), follows the result.

To prove \ref{ksparse2}, suppose that $g>0$ and $\ell_{g}(H)=2g-K$. If $H$ is be a $\kappa$-sparse numerical semigroup, we have that $\sum\limits_{m=1}^{\kappa}v_{m}(H)=g$ and, by the item \ref{ksparse1} above, $2g-K=\sum\limits_{m=1}^{\kappa}mv_{m}(H)-1$. So, $\sum\limits_{m=1}^{\kappa}mv_{m}(H)=2g-K+1$ and we have
    $$
\sum_{m=2}^{\kappa}(m-1)v_{m}(H)=2g-K+1-\sum_{m=1}^{\kappa}v_{m}(H)=g-K+1.
    $$
The converse implication is clear.
    \end{proof}

Let $\ss$\, be the collection of all numerical semigroups$.$

    \begin{theorem}\label{thm5.3}
\hfill
    \begin{enumerate}[label=$(\arabic*)$,leftmargin=*,align=left,start=1,itemsep=0pt]
\item The sequence $\left(\ss_{\kappa}\right)_{\kappa\in\NN}$ is an strictly ascending chain$.$

\item $\ss_{\kappa_1}^{\; \ast} \cap \ss_{\kappa_2}^{\; \ast} = \emptyset,$ for all positive integers $\kappa_{1}$ and $\kappa_{2}$ such that $\kappa_{1}\neq \kappa_{2}.$

\item $\ss= \bigcup\limits_{\kappa\in\NN}\ss_{\kappa}=\bigcupdot\limits_{\kappa\in\NN}\ss_{\kappa}^{\;\ast}.$

\item $\bigcap\limits_{\kappa\in\NN}\ss_{\kappa}=\{\NNo\}.$
    \end{enumerate}
    \end{theorem}

    \begin{proof}
We show the item (1). If $\kappa=1$, by Lemma \ref{lem5.1}, is clear that $\ss_{1}\subsetneq\ss_{2}$. If $\kappa\geq 2$, by Theorem \ref{thm4.3}, follows that $\ss_{\kappa}\subset\ss_{\kappa+1}$ and the Example \ref{ex4.5} show that  $\ss_{\kappa}\neq\ss_{\kappa+1}$.

The item (2), follows directly from the definition of pure $\kappa$-sparse numerical semigroup.

To prove (3), note that $\bigcup\limits_{\kappa\in\NN}\ss_{\kappa}\subset\ss$. Now, let $H\in\ss$\, with genus $g$. If $g=0$, then $H=\NNo\in\ss_{1}$. In another case, if $g>0$, taking $\kappa\:=\max\{\ell_{i}-\ell_{i-1}:\,1\leq i\leq g\}$, follows that $H\in\ss_{\kappa}$. So, in any case, $\ss\,\subset\bigcup\limits_{\kappa\in\NN}\ss_{\kappa}$. The second equality is obvious.

The item (4) is trivial, by item (1) and Lemma \ref{lem5.1} item (i).
    \end{proof}




The item (3) in the previous theorem allows to define an equivalence relation $\sim$ on $\ss$ in the following way: given $H_{1},H_{2}\in\ss$,
    $$
\text{$H_{1}\sim H_{2}\;\Leftrightarrow\;$ there is $\kappa\in\NN$ such that $H_{1}$ and $H_{2}$ are in}\, \ss_{\kappa}^{\;\ast}.
    $$

Now, we are interested in the Frobenius varieties, concept introduced by J. C. Rosales in \cite{rosales3}. We will show that the families of $\kappa$-sparse numerical semigroups, for $\kappa\geq 2$, are examples of Frobenius varieties. Remember that the intersection of two numerical semigroups is a numerical semigroup and if $H$ is a numerical semigroup of genus $g>0$, then $H\cup\{\ell_{g}(H)\}$ is also a numerical semigroup (see \cite{rosales2} for details).


    \begin{definition}\label{def5.4}
A Frobenius variety is a nonempty set $\vv$\, of numerical semigroups fulfilling the following conditions:
    \begin{enumerate}[label=(\roman*),leftmargin=*,align=left,start=1,itemsep=0pt]
\item for all $H_{1},H_{2}\in\vv$\,, $H_{1}\cap H_{2}\in\vv$\,;

\item for all $H \in \vv$\, with genus $g>0$, $H\cup\{\ell_{g}(H)\}\in \vv$\,.
    \end{enumerate}
    \end{definition}

    \begin{lemma}\label{lem5.5}
Let $\kappa$ an integer such that $\kappa\geq2.$ If $H_{1},H_{2}\in\ss_{\kappa},$ then $H_{1}\cap H_{2}\in\ss_{\kappa}.$
    \end{lemma}

    \begin{proof}
Let $H_{1}$ and $H_{2}$ are in $\ss_{\kappa}$, with genus $g_{1}$ and $g_{2}$, respectively. It is clear that if $g_{1}=0$ or $g_{2}=0$, then $H_{1}\cap H_{2}\in\ss_{\kappa}$.

Thus, suppose that $g_{1}>0$ and $g_{2}>0$. Let $\gaps(H_{1})=\{\ell_{1}(H_{1})<\dots<\ell_{g_{1}}(H_{1})\}$ and $\gaps(H_{2})=\{\ell_{1}(H_{2})<\dots<\ell_{g_{1}}(H_{2})\}$. If $g$ is the genus of $H_{1}\cap H_{2}$, then $\gaps(H_{1}\cap H_{2})=\{\ell_{1}<\dots<\ell_{g}\}$, where $\{\ell_{1},\dots,\ell_{g}\}=\{\ell_{1}(H_{1}),\dots,\ell_{g_{1}}(H_{1})\}\cup\{\ell_{1}(H_{2}),\dots,\ell_{g_{2}}(H_{2})\}$.

Note that, if $g=1$, we have that $g_{1}=g_{2}=1$ and then $H_{1}\cap H_{2}=\NN_{1}\in\ss_{\kappa}$. For $g>1$, let $i\in\{2,\dots,g\}$. If $\ell_{i},\ell_{i-1}\in\gaps(H_{r})$, for some $r\in\{1,2\}$, then $\ell_{i}-\ell_{i-1}\leq\kappa$. Now, if $\ell_{i}\in\gaps(H_{r})$ and $\ell_{i-1}\in\gaps(H_{s})$, for $r,s\in\{1,2\}$ with $r\neq s$, then there is $j\in\{1,\dots,g_{2}\}$ such that $\ell_{i-1}=\ell_{j}(H_{s})$, since $\ell_{i-1}\in\gaps(H_{s})$. So, if $\ell_{i}\in\gaps(H_{s})$, we have that $\ell_{i}=\ell_{j+1}(H_{s})$ and follows that $\ell_{i}-\ell_{i-1}=\ell_{j+1}(H_{s})-\ell_{j}(H_{s})\leq\kappa$. On the other hand, if $\ell_{i}\notin\gaps(H_{s})$, we have that $\ell_{i}<\ell_{j+1}(H_{s})$ and follows that $\ell_{i}-\ell_{i-1}<\ell_{j+1}(H_{s})-\ell_{j}(H_{s})\leq\kappa$. Therefore, we have that $H_{1}\cap H_{2}$ is a $\kappa$-sparse numerical semigroup, that is, $H_{1}\cap H_{2}\in\ss_{\kappa}.$
    \end{proof}

    \begin{proposition}\label{prop5.6}
Let $\kappa$ be an integer such that $\kappa\geq2.$ The set $\ss_{\kappa}$ is a Frobenius variety$.$
    \end{proposition}

    \begin{proof}
The condition (i) in the Definition \ref{def5.4} is satisfied by the previous lemma. Now, we will show that the condition (ii) is also satisfied. In fact, let $H\in\ss_{\kappa}$ with genus $g>0$ and $\gaps(H)=\{\ell_{1}<\cdots<\ell_{g}\}$. If $g=1$, then $H\cup\{\ell_{1}\}=\NNo\in\ss_{\kappa}$. If $g>1$, then the genus of $H\cup\{\ell_{g}\}$ is equal to $g-1$ and $\gaps(H\cup\{\ell_{g}\})=\{\ell_{1}<\cdots<\ell_{g-1}\}$. So, by Theorem \ref{thm4.3} item (ii), follows that $H\cup\{\ell_{g}\}\in\ss_{\kappa}$.
    \end{proof}

The Figure \ref{fig1}, illustrates a retrospective of the $\kappa$-sparse numerical semigroups as a generalization of the sparse and Arf numerical semigroups.

  \begin{figure}[!h]
 \centering
\begin{tikzpicture}[font=\small,
rnd/.style={draw=#1,rounded corners=8pt,align=center,text width=4.4cm,minimum height=1.2cm},
ar/.style={->,>=latex},node distance=0.5cm and 3cm
]
\node[rnd,white] (1) {};
\node[rnd,below=of 1] (purek1) {Pure $(\kappa+1)$-sparse \\ numerical semigroups};
\node[rnd,below=of purek1] (purek) {Pure $\kappa$-sparse numerical \\ semigroups};

\node[rnd,right=of 1] (ns) {Numerical semigroups};
\node[rnd,below=of ns] (k1sparse){$(\kappa+1)$-sparse \\ numerical semigroups};
\node[rnd,below=of k1sparse] (ksparse) {$\kappa$-sparse numerical \\ semigroups};
\node[rnd,below=of ksparse] (sparse) {Sparse numerical \\ semigroups};
\node[rnd,below=of sparse] (arf) {Arf numerical semigroups};
\node[rnd,below=of arf] (ordinary) {Ordinary semigroups};
\node[rnd,below=of ordinary] (trivial) {Trivial numerical \\ semigroup: $\NNo$};

\draw[ar] (k1sparse) -- (ns);
\draw[ar] (ksparse) -- (k1sparse);
\draw[ar] (sparse) -- (ksparse);
\draw[ar] (arf) -- (sparse);
\draw[ar] (ordinary) -- (arf);
\draw[ar] (trivial) -- (ordinary);
\draw[ar] (purek1) -- (k1sparse);
\draw[ar] (purek) -- (ksparse);

\end{tikzpicture}
\caption{Diagram of $\kappa$-sparse numerical semigroups}
\label{fig1}
\end{figure}
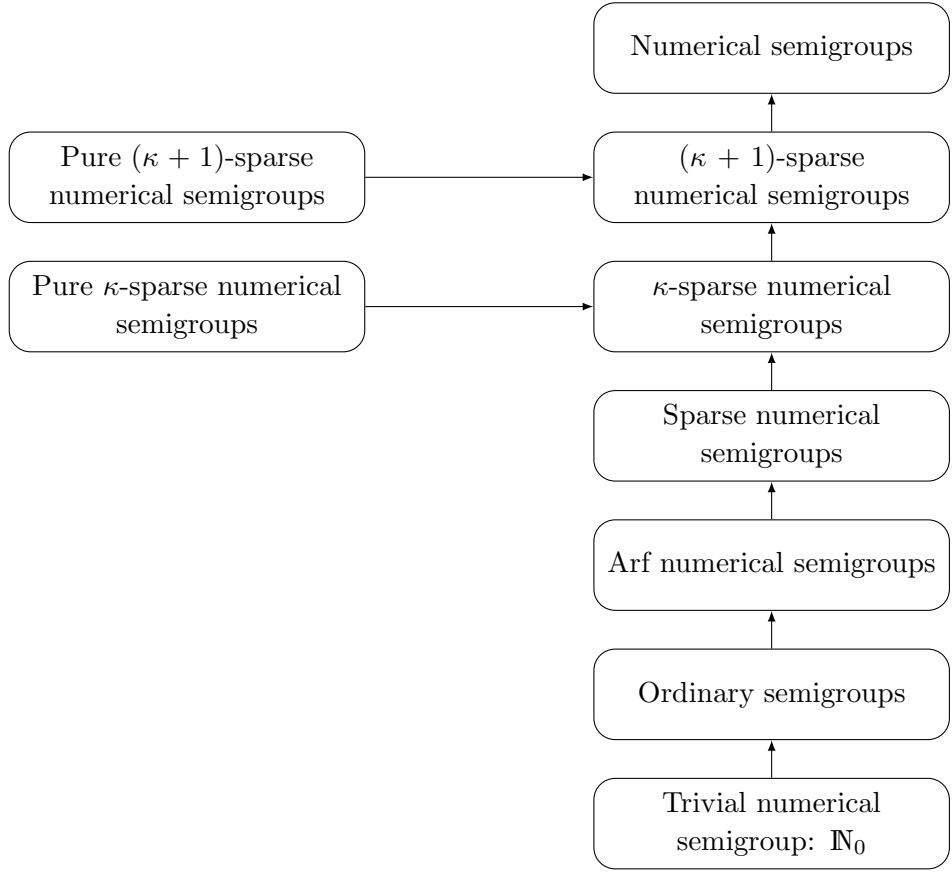


    \bibliography{mybibfile}

\end{document}